\documentclass[12pt, twoside, leqno]{article}

\usepackage[pagewise]{lineno} 

\usepackage{amsmath,amsthm}
\usepackage{amssymb}

\usepackage{enumitem}

\usepackage{graphicx}

\usepackage[T1]{fontenc}

\newtheorem{theorem}{Theorem}[section]

\newtheorem{lemma}[theorem]{Lemma}

\theoremstyle{definition}
\newtheorem{definition}[theorem]{Definition}
\newtheorem{remark}[theorem]{Remark}

\numberwithin{equation}{section}

\begin{document}


\baselineskip=17pt


\title{Super-Biderivations and Linear Super-Commuting Maps on the Lie Superalgebras}

\author{\textsc{Liming Tang}\\
\small{School of Mathematical Sciences}\\
\small{Harbin Normal University}\\
\small{150025 Harbin, China}\\
\small{E-mail: limingtang@hrbnu.edu.cn}
\and
\textsc{Lingyi Meng}\\
\small{School of Mathematical Sciences}\\
\small{Harbin Normal University}\\
\small{150025 Harbin, China}\\
\small{E-mail: 1548443934@qq.com}
\and
\textsc{Liangyun Chen}\footnote{corresponding author}\\
\small{School of Mathematics and Statistics}\\
\small{Northeast Normal University}\\
\small{130024 Changchun, China}\\
\small{E-mail: chenly640@nenu.edu.cn}}

\date{}

\maketitle


\renewcommand{\thefootnote}{}


\renewcommand{\thefootnote}{\arabic{footnote}}
\setcounter{footnote}{0}
\begin{abstract}
Suppose the ground field is algebraically closed and of characteristic different from 2. In this paper, we described the intrinsic connections among linear super-commuting maps, super-biderivations and  centroids  for Lie superalgebras satisfying certain assumptions. This  is a generalization of the results of Bre$\check{s}$ar and Zhao on Lie algebras.
\end{abstract}

\section{Introduction}
Throughout we work over
  an algebraically closed field $\mathbb{F}$ of characteristic different from 2 and all vector
spaces and algebras are over $\mathbb{F}$. The aim is to describe the interactions among linear super-commuting maps, super-biderivations and  centroids  for Lie superalgebras satisfying certain conditions. Most results of commuting maps and  biderivations on Lie algebras are from these papers \cite{LL00,WYC11,WY13,C16,HWX16,T16,GLZ18,BZ18}. The idea of this paper springs from biderivations and linear commuting maps on the Lie algebras \cite{BZ18}.
 Bre$\check{s}$ar and Zhao \cite{BZ18} presented two results for a perfect and centerless Lie algebra $L$. One is that  every skew-symmetric biderivation  is of the form of the centroid of $L.$ The other is that  every commuting linear map from $L$ to $L$ lies in the centroid of $L$ under a strong assumption. The parallel results are obtained for Lie superalgebras in this paper. Our reason for doing so goes beyond a pure generalization. Rather, we expect that linear commuting maps on Lie superalgebras will play an important role in the theory of functional identities on
Lie superalgebras. The reason  is a fact that the description of linear commuting maps on Lie algebras can
be viewed as a testing case for developing the theory of functional identities on
Lie algebras \cite{BZ18}. The root lies in additive commuting
maps on prime rings \cite{B93} which eventually led to the theory of functional identities on
noncommutative rings \cite{BCM07}.

An interest in studying super-biderivations and super-commuting maps on superalgebras has grown more recently \cite{XWH16,YT17, FD17,YCC18, CS19}.
Xia, Wang and Han \cite{XWH16}  proved that any super-skewsymmetric super-biderivation of Super-Virasoro Algebras is inner and obtained the form of every linear super-commuting map. Yuan and Tang \cite{YT17} characterized super-biderivations of classical simple Lie superalgebras over the complex field $\mathbb{C}$  and proved that all super-biderivations of classical simple Lie superalgebras are inner super-biderivations. Fan and Dai \cite{FD17}  proved
that all super-biderivations on the centerless super-Virasoro algebras are inner super-biderivations and studied
the linear super-commuting maps on the centerless super-Virasoro algebras.
Yuan, Chen and Cao \cite{YCC18} characterized the super-biderivations of Cartan
type Lie superalgebras over the complex field $\mathbb{C}$ and  proved that
all super-biderivations of Cartan type simple Lie superalgebras are inner super-biderivations.
Cheng and Sun \cite{CS19}  proved that every super-skewsymmetric super-biderivation of the twisted $N=2$ superconformal algebra is inner. The results in these papers \cite{XWH16,FD17,YT17,YCC18,CS19} hold for  certain superalgebras by using the approach depending on the certain superalgebras, respectively. In this paper, our results hold for the general Lie superalgebras satisfying certain assumptions. The results are suitable for simple Lie superalgebras.

The paper is structured as follows: Section 1 will introduce the definitions of super-biderivations, super-commuting maps and centroids  for a Lie superalgebra. Section 2 will be devote to  describing the internal connections between super-biderivations and centroids. Finally, in Section 3, the connections between super-commuting maps and centroids will be expressed.

\section{Definitions and notions}
Suppose $L=L_{\bar{0}}\oplus L_{\bar{1}}$ is a superalgebra over a field $\mathbb{F}$. If a linear map $f:L\longrightarrow L$ such that $$f(L_{i})\subseteq L_{i+\tau}, \forall\, i\in \mathbb{Z}_{2},$$
then $f$ is \textit{a homogeneous linear map} of degree $\tau$, that is $|f|=\tau$, where $\tau \in \mathbb{Z}_{2}$. If $\tau=\bar{0}$, then $f$ is called \textit{an even linear map}. For a linear  map $f$, if $|f|$ occurs, then $f$ implies a homogeneous map. For a Lie superalgebra $L=L_{\bar{0}}\oplus L_{\bar{1}}$, if $|x|$ occurs, then $x$ implies a homogeneous element in $L$. A linear map $D:L\longrightarrow L$ is \textit{a superderivation} of $L$ \cite{K77, S76}, if
$$D([x,y])=[D(x), y]+(-1)^{|D||x|}[x, D(y)], \forall\, x,y\in L.$$
Denote by ${\rm{Der}}_{\tau}(L)$  the set of all superderivations of degree $\tau$ for $L$. Obviously, ${\rm{Der}}(L)={\rm{Der}}_{\bar{0}}(L)\oplus{\rm{Der}}_{\bar{1}}(L)$.
For a Lie superalgebra $L=L_{\bar{0}}\oplus L_{\bar{1}}$, if a bilinear map
$$\delta:L\times L\longrightarrow L$$satisfying that$$\delta(L_{i},L_{j})\subseteq L_{i+j+\tau}, \forall\, i,j \in \mathbb{Z}_{2},$$
then $\delta$ is a homogeneous bilinear map of degree $\tau$, that is $|\delta|=\tau$, where $\tau\in \mathbb{Z}_{2}$.
A bilinear map $\delta:L\times L\longrightarrow L$ is \textit{a super-biderivaton} of $L$ \cite{BZ18}, if
\begin{eqnarray}\nonumber
&&\delta(x,y)=-(-1)^{|x||y|}\delta(y,x),\\ \nonumber
&&\delta([x,y],z)=(-1)^{|\delta||x|}[x, \delta(y,z)]+(-1)^{|y||z|} [\delta(x,z), y],  \forall \, x,y,z\in L.\nonumber
\end{eqnarray}
Denote by $\mathrm{BDer}_{\tau}(L)$  the set of all super-biderivations of degree $\tau$ for $L$. Obviously, $$\mathrm{BDer}(L)=\mathrm{BDer}_{\bar{0}}(L)\oplus\mathrm{BDer}_{\bar{1}}(L).$$

\begin{definition}\label{s}
For a Lie superalgebra $L=L_{\bar{0}}\oplus L_{\bar{1}}$, an even linear map  $f:L\longrightarrow L$ is called \textit{a linear super-commuting map} if
$$[f(x), y]=[x, f(y)],\forall\, x,y\in L.$$
\end{definition}

\section{Super-biderivations}
\begin{lemma}\label{superde}

 Let $L$ be a Lie superalgebra and $f:L\longrightarrow L$ be a linear super-commuting map. Suppose $$\delta(x,y)=[x, f(y)], \forall\, x,y\in L.$$ Then $\delta:L\times L\longrightarrow L$ is a super-biderivation.
\end{lemma}
\begin{proof}
 By $$\delta(x,y)=[x, f(y)], \forall\, x,y\in L,$$ it follows that $|\delta|=|f|=\bar{0}$.
For one thing, we have
\begin{eqnarray}\nonumber
-(-1)^{|x||y|}\delta(y,x)&=&-(-1)^{|x||y|}[y, f(x)]\\ \nonumber
&=&(-1)^{|x||y|+|y|(|f|+|x|)}[f(x), y]\\ \nonumber
&=&[f(x), y]\\ \nonumber
&=&[x, f(y)]\\ \nonumber
&=&\delta(x,y).\nonumber
\end{eqnarray}
For the other, we get
\begin{eqnarray}\nonumber
\delta([x,y],z)&=&[[x,y], f(z)]\\ \nonumber
&=&[x,[y, f(z)]]-(-1)^{|x||y|}[y,[x, f(z)]]\\ \nonumber
&=&[x,\delta(y,z)]-(-1)^{|x||y|}[y,\delta(x,z)]\\ \nonumber
&=&[x, \delta(y,z)]+(-1)^{|x||y|+|y|(|\delta|+|x|+|z|)}[\delta(x,z), y]\\ \nonumber
&=&(-1)^{|\delta||x|}[x, \delta(y,z)]+(-1)^{|y||z|}[\delta(x,z), y].\nonumber
\end{eqnarray}
Combining (1) with (2), then $\delta$ is a super-biderivation.
\end{proof}
\begin{definition}\cite{XJ10}\label{cent}
 Let $L$ be a Lie superalgebra and $\gamma:L\longrightarrow L$  a linear map. Then
$${\rm{\Gamma}}(L)=\{\gamma:L\rightarrow L \mid  \gamma([x,y])=(-1)^{|\gamma||x|}[x,\gamma(y)], \forall\, x,y\in L\}$$ is called \textit{the centroid} of $L$. Denote $\mathrm{\Gamma}_{\tau}(L)$ by the set of all elements of degree  $\tau$ in  ${\rm{\Gamma}}(L)$. Obviously, $$\mathrm{\Gamma}(L)=\mathrm{\Gamma}_{\bar{0}}(L)\oplus\mathrm{\Gamma}_{\bar{1}}(L).$$
\end{definition}
  \begin{lemma}\label{2.4}
  Let $L$ be
  a Lie superalgebra. Suppose that $\gamma:L\longrightarrow L$ is a 
  linear map and $\delta:L\times L\longrightarrow L$  a bilinear map. Let $\delta(x,y)=\gamma([x,y])$. If $\gamma\in \mathrm{\Gamma}(L)$, then $\delta$ is a super-biderivation.
\end{lemma}
\begin{proof}
By the definition \ref{cent}, it follows that
$$\delta(x,y)=\gamma([x,y])=(-1)^{|\gamma||x|}[x,\gamma(y)], \forall\, x, y\in L.$$
Then $|\delta|=|\gamma|$.
First, we have\begin{eqnarray}\nonumber
-(-1)^{|x||y|}\delta(y,x)&=&-(-1)^{|x||y|}\gamma([y,x])\\ \nonumber
&=&-(-1)^{|x||y|}\gamma(-(-1)^{|x||y|}[x,y])\\ \nonumber
&=&\delta(x,y),  \forall\, x, y\in L.\nonumber
\end{eqnarray}
Second, by super-Jacobi identity we have
\begin{eqnarray}\nonumber
\delta([x,y],z)&=&\gamma([[x,y],z])\\ \nonumber
&=&\gamma([x,[y,z]]-(-1)^{|x||y|}[y,[x,z]])\\ \nonumber
&=&\gamma([x,[y,z]])-(-1)^{|x||y|}\gamma([y,[x,z]])\\ \nonumber
&=&(-1)^{|x||\gamma|}[x, \gamma([y,z])]-(-1)^{|x||y|+|\gamma||y|}[y, \gamma([x,z])]\\ \nonumber
&=&(-1)^{|x||\gamma|}[x, \gamma([y,z])]+(-1)^{|x||y|+|\gamma||y|+|y|(|\gamma|+|x|+|z|)}[\gamma([x,z]), y]\\ \nonumber
&=&(-1)^{|x||\gamma|}[x, \gamma([y,z])]+(-1)^{|y||z|}[\gamma([x,z]), y]\\ \nonumber
&=&(-1)^{|x||\delta|}[x, \delta(y,z)]+(-1)^{|y||z|}[\delta(x,z), y], \forall\, x, y, z \in L.\nonumber
\end{eqnarray}
Then $\delta$ is a super-biderivation.
\end{proof}
\begin{definition}
 Suppose $S$ is a non-empty subset of Lie superalgebra $L$. Let $$\mathrm{Z}_{L}(S)=\{v\in L \mid [S, v]=0\}.$$
  If $S=L$,  then $\mathrm{Z}(L)=\mathrm{Z}_{L}(L)$ is the center of $L$. If $\mathrm{Z}_{L}(L)=\{0\}$, then $L$ is centerless. Denote $L'=[L,L]$ by the derived algebra of $L$. If $L=L'$, then $\mathrm{Z}_{L}(L')=\mathrm{Z}_{L'}(L')$ is the center of $L'$.
\end{definition}
\begin{lemma}\cite{FD17}\label{249}
Suppose $L$ is a Lie superalgebra and  $\delta:L\times L\longrightarrow L$ is a super-biderivation of $L$. Then
$$[\delta(x,y),[u,v]]=(-1)^{|\delta|(|x|+|y|)}[[x,y], \delta(u,v)],\forall\, x,y,u,v\in L.$$
\end{lemma}
\begin{lemma}\label{2.6}
Suppose $L$ is a Lie superalgebra and $\delta:L\times L\longrightarrow L$ is a super-biderivation. Then
$$\delta(u,[x,y])-(-1)^{|\delta||u|}[u, \delta(x,y)]\in {\rm{Z}}_{L}(L'),\forall\, u,x,y\in L.$$
\end{lemma}
\begin{proof}
By Lemma \ref{249}, we have
\begin{align}\label{2.1}
[\delta(z,w), [x,y]]=(-1)^{|\delta|(|z|+|w|)}[[z,w], \delta(x,y)].
\end{align}
On the one hand, replacing $x$ by $[x,u]$, then
\begin{align}\label{250}
[\delta(z,w), [[x,u],y]]=(-1)^{|\delta|(|z|+|w|)}[[z,w], \delta([x,u],y)].
\end{align}
On the other hand, by super-Jacobi identity we have
\begin{align}\label{251}
[\delta(z,w), [[x,u],y]]=[\delta(z,w),[x,[u,y]]]-(-1)^{|x||u|}[\delta(z,w),[u,[x,y]]].
\end{align}
Therefore, by (\ref{250}),
\begin{eqnarray}\label{252}
&&[\delta(z,w), [[x,u],y]]\\ \nonumber
&=&(-1)^{|\delta|(|z|+|w|)}[[z,w], \delta(x,[u,y])]-(-1)^{|x||u|+|\delta|(|z|+|w|)}[[z,w],\delta(u,[x,y])].\nonumber
\end{eqnarray}
Comparing  (\ref{251}) and (\ref{252}) we have
\begin{align}\label{301}
(-1)^{|\delta|(|z|+|w|)}[[z,w],(\delta([x,u],y)]-\delta(x,[u,y])+(-1)^{|x||u|}\delta(u,[x,y]))=0.
\end{align}
Because $\delta$ is a super-biderivation, by (\ref{301}) we get
$$2(-1)^{|\delta|(|z|+|w|)+|x||u|}[[z,w],([\delta(u,x), y]+(-1)^{|\delta||x|}[x, \delta(u,y)]-(-1)^{|\delta||x|+|x||u|}[u, \delta(x,y)])]=0.$$
According to $$[\delta(u,x), y]+(-1)^{|\delta||x|}[x, \delta(u,y)]=\delta(u,[x,y]),$$
it follows that
$$(-1)^{|\delta||x|+|x||u|}\delta(u,[x,y])-(-1)^{|\delta||u|}[u, \delta(x,y)]\in {\rm{Z}}_{L}(L'),\forall\, u, x,y\in L.$$
\end{proof}
\begin{lemma}\label{2.7}
Suppose  $L$ is a Lie superalgebra with $L=L'$ and $\delta:L\times L\longrightarrow L$ is a super-biderivation.  Then
\begin{align}\label{3.25}
\delta(u,[x,y])=(-1)^{|\delta||u|}[u, \delta(x,y)],\forall\, u, x,y\in L.
\end{align}
\end{lemma}
\begin{proof}
 We have
\begin{eqnarray}\nonumber
&&\delta([u,v],[x,y])-(-1)^{|\delta|(|u|+|v|)}[[u,v], \delta(x,y)]\\ \nonumber
&=&(-1)^{|\delta||u|}[u, \delta(v,[x,y])]+(-1)^{|v|(|x|+|y|)}[\delta(u,[x,y]), v]\\ \nonumber
&&-(-1)^{|\delta|(|u|+|v|)}[u,[v, \delta(x,y)]]+(-1)^{|\delta|(|u|+|v|)+|u||v|}[v,[u, \delta(x,y)]]\\ \nonumber
&=&(-1)^{|\delta||u|}[u,\delta(v,[x,y])-(-1)^{|\delta||v|}[v, \delta(x,y)]]\\ \nonumber
&&-(-1)^{|v|(|x|+|y|)+|v|(|\delta|+|u|+|x|+|y|)}[v, \delta(u,[x,y])]+(-1)^{|\delta|(|u|+|v|)+|u||v|}[v,[u, \delta(x,y)]]\\ \nonumber
&=&(-1)^{|\delta||u|}[u,\delta(v,[x,y])-(-1)^{|\delta||v|}[v, \delta(x,y)]]\\ \nonumber
&&-(-1)^{|v|(|\delta|+|u|)}[v, \delta(u,[x,y])-(-1)^{|\delta||u|}[u, \delta(x,y)]], \forall\, u,v,x,y\in L. \nonumber
\end{eqnarray}
By Lemma \ref{2.6} and $L=L'$ we have $$\delta(u,[x,y])=(-1)^{|\delta||u|}[u, \delta(x,y)],\forall\, u, x,y\in L.$$
\end{proof}


\begin{theorem}
Suppose that $L$ is a centerless Lie superalgebra with $L=L'$. Then every superderivation  $\delta:L\times L\longrightarrow L$ can be written as $$\delta(x,y)=\gamma([x,y]),$$ where $\gamma$ is in the centroid of $L$.
\end{theorem}

\begin{proof}
Suppose $\gamma:L\longrightarrow L$ is a linear map given by
\begin{align}\label{2.5}
\gamma([x,y])=\delta(x,y),\forall\, x,y\in L.
\end{align}
By Lemma \ref{2.7}, $\gamma$ is well defined. In fact, suppose $\Sigma_{i}[x_{i},y_{i}]=0$, we have
$$0=\delta(u,\Sigma_{i}[x_{i},y_{i}])=\Sigma_{i}\delta(u,[x_{i},y_{i}])=(-1)^{|u||\delta|}[u, (\Sigma_{i}\delta(x_{i},y_{i}))].$$
By $\mathrm{Z}_{L}(L)=\{0\}$, we have $\Sigma_{i}\delta(x_{i},y_{i})=0.$
By (\ref{3.25}),  we have $$\delta(u,v)=(-1)^{|\delta||u|}[u, \gamma(v)],\forall\, u,v\in L.$$
By (\ref{2.5}), we get $$\gamma([x,y])=(-1)^{|\delta||u|}[x, \gamma(y)],\forall\, x,y\in L.$$
Then it suggests that
$\gamma \in \mathrm{\Gamma}(L)$.
\end{proof}

\begin{definition}
Suppose that $\delta$ is a skew-symmetric  bilinear map  with $\delta(L,L')=0$.  Then $\delta$  is a super-biderivation, called \textit{a trivial super-biderivation} of $L$.
\end{definition}
\begin{remark}
 Suppose $\delta:L\times L\longrightarrow L$ is an arbitrary super-biderivation of $L$. 
Then we have
\begin{eqnarray}\nonumber
0&=&\delta([z,x],y)\\ \nonumber
&=&(-1)^{|\delta||z|}[z,\delta(x,y)]+(-1)^{|x||y|}[\delta(z,y),x]\\ \nonumber
&=&(-1)^{|\delta||z|}[\delta(z,y),x],  \forall\, x,y\in L,z\in \mathrm{Z}(L),
\end{eqnarray}
that is, $[\delta(z,y),x]=0.$ So
$\delta(\mathrm{Z},L)\subset \mathrm{Z}(L)$. Furthermore, let $\bar{L}=L/\mathrm{Z}(L)$ and
$$\bar{\delta}(\bar{x},\bar{y})=\overline{\delta(x,y)}, \forall \,x,y \in L.$$ Then $\bar{\delta}:\bar{L}\times \bar{L}\longrightarrow \bar{L}$ is a super-biderivation, where $ \bar{x}=x+\mathrm{Z}(L)\in \bar{L},\forall\, x\in L$.
\end{remark}
\begin{theorem}
Suppose $L$ is a Lie superalgebra. Up to isomorphism, the map $\delta \longrightarrow \bar{\delta}$ is one-to-one correspondence from a trivial superderivation $\delta$ of  $L$ to a trivial superderivation $\bar{\delta}$ of $\bar{L}$.
\end{theorem}
\begin{proof}
Obviously, it is surjective.  It remains to show that the map is injective. Suppose $\delta_{1}$ and $\delta_{2}$ are superderivations of $L$ satisfying $\bar{\delta_{1}}=\bar{\delta_{2}}$. Let $\delta=\delta_{1}-\delta_{2}$.
Then
$$\bar{\delta_{1}}(\bar{x},\bar{y})=\bar{\delta_{2}}(\bar{x},\bar{y}), \forall\, \bar{x},\bar{y}\in \bar{L}.$$
By virtue of $$\bar{\delta_{1}}(\bar{x},\bar{y})=\overline{\delta_{1}(x,y)}\,  \mbox{and}\,  \delta_{2}(\bar{x},\bar{y})=\overline{\delta_{2}(x,y)},$$ we have
$$\overline{\delta_{1}(x,y)}=\overline{\delta_{2}(x,y)}.$$
Thus, $$\delta(x,y)=\delta_{1}(x,y)-\delta_{2}(x,y)\in \mathrm{Z}(L).$$ Then $\delta$ is a trivial super-biderivation.
\end{proof}

\begin{definition}
Suppose $L$ is a Lie superalgebra. If $\delta:L\times L\longrightarrow \mathrm{Z}_{L}(L')$ is a super-biderivation of $L$ satisfying $\delta(L', L')=0$, then $\delta$ is called \textit{a special super-biderivation} of $L$.
\end{definition}
\begin{remark}
Suppose $L$ is a Lie superalgebra and $\delta:L\times L\longrightarrow L$ is a super-biderivation satisfying
\begin{align}\label{2.8}
\delta(u,[x,y]))=[\delta(u,x),y]+(-1)^{(|\delta|+|u|)|x|} [x,\delta(u,y)],\forall\, x,y,u\in L.
\end{align}
 Let $$\delta':=\delta|_{L'}:L'\times L'\longrightarrow L'.$$ Then $\delta'$ is a super-biderivation of $L'$.
\end{remark}
\begin{theorem}
Let $L$ be a centerless Lie superalgebra. Then

{\rm{(1)}} Up to isomorphism, every special superbidevition of $L$ is the unique extention of a special super-biderivation of $L'$.

{\rm{(2)}} If $L$ is a Lie superalgebra with $L=L'$, then any special super-biderivation of $L$ is zero.
\end{theorem}
\begin{proof}
(1) Suppose $\delta_{1}$ and $\delta_{2}$ are  super-biderivations of $L$ satisfying  $\delta_{1}'=\delta_{2}'$. Let $\delta=\delta_{1}-\delta_{2}$.  Then $\delta(L',L')=0$. Substitute $u,y\in L'$ into (\ref{2.8}). We have
$$[\delta(u,x),y]=0,\forall x\in L,y,u\in L',$$
that is $\delta(L,L')\subset \mathrm{Z}_{L}(L')$. By Lemma \ref{2.7}, one can get $[L,\delta(L,L)]\subset\mathrm{Z}_{L}(L').$

By (\ref{2.1}), we have $$0=[[L,L],\delta(L,L')]=[\delta(L,L),[L,L']].$$
Then, for any $ x,y,z,u,v\in L$, it suggests that
\begin{eqnarray}\nonumber
0&=&[[[x,y],z],\delta(u,v)]\\ \nonumber
&=&[[x,y],[z,\delta(u,v)]]-(-1)^{(|x|+|y|)|z|}[z,[[x,y],\delta(u,v)]]\\ \nonumber
&=&(-1)^{(|x|+|y|)|z|+|z|(|x|+|y|+|\delta|+|u|+|v|)}[[[x,y],\delta(u,v)],z]\\ \nonumber
&=&(-1)^{|z|(|\delta|+|u|+|v|)}[[[x,y],\delta(u,v)],z].\nonumber
\end{eqnarray}Since $L$ is centerless, one gets $\delta(L,L)\subset\mathrm{Z}_{L}(L')$. Then $\delta$ is a special super-biderivation of $L$.

(2) Suppose $\delta$ is a special super-biderivation of $L$. By (1) we have $\delta(L,L')\subset \mathrm{Z}_{L}(L')$. Since $L$ is a centerless Lie superalgebra and $L=L'$. Then $\delta(L,L')=0$.

Suppose $\delta \neq0$. Then there exist $x_{1}$ and $x_{2}\in L$ satisfying $\delta(x_{1},x_{2})=z_{12}\neq0$. Since $L$ is centerless. Picking $0\neq x_{3}\in L$ such that $[x_{3},z_{12}]=z\neq0$.
Let $\delta(x_{i},x_{j})=z_{ij}$, where $i,j=1,2,3$. For one thing, we have
\begin{eqnarray}\nonumber
0&=&\delta([x_{1},x_{3}],x_{2})\\ \nonumber
&=&(-1)^{|\delta||x_{1}|}[x_{1},\delta(x_{3},x_{2})]+(-1)^{|x_{2}||x_{3}|}[\delta(x_{1},x_{2}),x_{3}]\\ \nonumber
&=&(-1)^{|\delta||x_{1}|}[x_{1},z_{32}]+(-1)^{|x_{2}||x_{3}|}[z_{12},x_{3}]\\ \nonumber
&=&(-1)^{|\delta||x_{1}|}[x_{1},z_{32}]-(-1)^{|x_{2}||x_{3}|+(|\delta|+|x_{1}|+|x_{2}|)|x_{3}|}[x_{3},z_{12}]\\ \nonumber
&=&(-1)^{|\delta||x_{1}|}[x_{1},z_{32}]-(-1)^{|x_{1}||x_{3}|+|\delta||x_{3}|}z. \nonumber
\end{eqnarray}\nonumber
For the other, we have
\begin{eqnarray}\nonumber
0&=&\delta([x_{1},x_{2}],x_{3})\\ \nonumber
&=&(-1)^{|\delta||x_{1}|}[x_{1},\delta(x_{2},x_{3})]+(-1)^{|x_{2}||x_{3}|}[\delta(x_{1},x_{3}),x_{2}]\\ \nonumber
&=&(-1)^{|\delta||x_{1}|}[x_{1},z_{23}]+(-1)^{|x_{2}||x_{3}|}[z_{13},x_{2}]. \nonumber
\end{eqnarray}\nonumber
For another, we have
\begin{eqnarray}\nonumber
0&=&\delta([x_{2},x_{3}],x_{1})\\ \nonumber
&=&(-1)^{|\delta||x_{2}|}[x_{2},\delta(x_{3},x_{1})]+(-1)^{|x_{1}||x_{3}|}[\delta(x_{2},x_{1}),x_{3}]\\ \nonumber
&=&(-1)^{|\delta||x_{2}|}[x_{2},z_{31}]+(-1)^{|x_{1}||x_{3}|}[z_{21},x_{3}]\\ \nonumber
&=&(-1)^{|\delta||x_{2}|}[x_{2},z_{31}]-(-1)^{|x_{1}||x_{3}|+|x_{1}||x_{2}|}[z_{12},x_{3}]\\ \nonumber
&=&(-1)^{|\delta||x_{2}|}[x_{2},z_{31}]+(-1)^{|x_{1}||x_{3}|+|x_{1}||x_{2}|+(|\delta|+|x_{1}|+|x_{2}|)|x_{3}|}[x_{3},z_{12}]\\ \nonumber
&=&(-1)^{|\delta||x_{2}|}[x_{2},z_{31}]+(-1)^{|x_{1}||x_{3}|+|x_{1}||x_{2}|+|\delta||x_{3}|+|\delta||x_{3}|}z.\nonumber
\end{eqnarray}\nonumber
Therefore, it follows at once
\begin{eqnarray}\nonumber
(-1)^{|x_{1}||x_{3}|+|\delta||x_{3}|}z&=&(-1)^{|\delta||x_{1}|}[x_{1},z_{32}]\\ \nonumber
&=&-(-1)^{|\delta||x_{1}|+|x_{2}||x_{3}|}[x_{1},z_{23}]=[z_{13},x_{2}]\\ \nonumber
&=&-(-1)^{|x_{1}||x_{3}|}[z_{31},x_{2}]\\ \nonumber
&=&(-1)^{|x_{1}||x_{3}|+(|\delta|+|x_{1}|+|x_{3}|)|x_{2}|}[x_{2},z_{31}]\\ \nonumber
&=&(-1)^{|\delta||x_{2}|+|x_{1}||x_{2}|+|x_{2}||x_{3}|+|x_{1}||x_{3}|++|x_{1}||x_{2}|+|x_{2}||x_{3}|+|\delta||x_{3}|+|\delta||x_{2}|}z\\ \nonumber
&=&-(-1)^{|x_{1}||x_{3}|+|\delta||x_{3}|}z. \nonumber
\end{eqnarray}\nonumber
Then we have $z=-z$, which contradicts to $z\neq0$, thus $\delta=0$.
\end{proof}





\section{Linear super-commuting maps}

\begin{lemma}\label{3.1}
Let $L$ be a Lie superalgebra. Suppose $f:L\longrightarrow L$ is a linear super-commuting map. Then
$$\big[[w,z],[u,f([x,y])-[x, f(y)]]\big]=0,\forall\, x,y,u,w,z\in L.$$
\end{lemma}
\begin{proof}
Since $f$ is a super-commuting linear map. Then we have $$[f(x), y]=[x, f(y)].$$ Let $$\delta:L\times L\longrightarrow L.$$ By Lemma \ref{2.4} then
$$\delta(x,y)=[x, f(y)]$$
is a super-biderivation. By Lemme \ref{2.6}, we have $|\delta|=|f|=\bar{0}$ and
$$\big[[w,z],\delta(u,[x,y])-(-1)^{|\delta||u|}[u, \delta(x,y)]\big]=0, \forall\, x,y,u,w,z\in L.$$
Since $$(-1)^{|\delta||u|}[u, \delta(x,y)]=(-1)^{|\delta||u|}[u, [x, f(y)]]=(-1)^{|f||u|}[u, [x, f(y)]]=[u, [x, f(y)]]$$ and $$\delta(u,[x,y])=(-1)^{|f||u|}[u, f([x,y])]=[u, f([x,y])], \forall\, x,y,u,w,z\in L.$$
 Then one can get
\begin{eqnarray}\nonumber
0&=&[[w,z],\delta(u,[x,y])-(-1)^{|\delta||u|}[u, \delta(x,y)]]\\ \nonumber
&=&[[w,z],[u, f([x,y])]-[u, [x, f(y)]]]\\ \nonumber
&=&[[w,z],[u,f([x,y])-[x, f(y)]]].\\ \nonumber
\end{eqnarray}\nonumber
Therefore, it follows that
$$[[w,z],[u,f([x,y])-[x, f(y)]]]=0.$$
\end{proof}

\begin{theorem}
 Let $L$ be a Lie superalgebra with $L=L'$ and $\mathrm{Z}_{L}(L')=\{0\}$. If $f:L\longrightarrow L$ is a linear super-commuting map, then $f\in \mathrm{\Gamma}_{\bar{0}}(L)$.
\end{theorem}
\begin{proof}
By Lemme \ref{3.1}, we have
$$[[w,z],[u,f([x,y])-[x, f(y)]]]=0,\forall\, x, y,z,u,w\in L.$$
then $$[u,f([x,y])-[x, f(y)]]\in \mathrm{Z}_{L}(L').$$
Since $\mathrm{Z}_{L}(L')=\{0\}$. It follows that
$$[u,f([x,y])-[x, f(y)]]=0.$$
Furthermore, we get $$f([x,y])-[x, f(y)]\in \mathrm{Z}_{L}(L)=\mathrm{Z}_{L}(L')=\{0\}.$$
that is, $$f([x,y])=[x, f(y)].$$
Thus $f\in \mathrm{\Gamma}_{\bar{0}}(L)$.
\end{proof}

\subsection*{Acknowledgements}
L. Tang was supported by the NSF of China (11971134), CSC, the NSF of HLJ Province (A2017005) and the NSF of HNU (XKB201403).
L. Chen was supported by  the NSF of China (11771069).

\end{document}